\documentclass{article}

\usepackage[utf8]{inputenc} 
\usepackage{url}

\usepackage{amsmath,amsfonts,amssymb}
\usepackage{amsthm}
\usepackage{graphicx}
\usepackage{enumerate}
\usepackage{multirow,bigdelim}
\usepackage{mathtools}
\usepackage{soul} 
\usepackage{pstricks}
\usepackage[margin=3cm]{geometry}
\usepackage{array}
\usepackage[normalem]{ulem}
\usepackage{url}
\usepackage{hyperref}
\usepackage{enumerate}
\usepackage{mathtools}
\hypersetup{
	colorlinks,
	linkcolor={red!50!black},
	citecolor={blue!50!black},
	urlcolor={blue!80!black}
}

\newcommand{\SOC}[2]{{\mathcal{L}^{#2} _{#1}}}

\newcommand{\conicHull}{\operatorname{cone}}

\newcommand{\norm}[1]{\lVert{#1}\rVert}
\newcommand{\inProd}[2]{\langle #1 , #2 \rangle }

\newcommand{\stdCone}{ {\mathcal{K}}}

\newcommand{\stdFace}{ \mathcal{F}}

\renewcommand{\Re}{\mathbb{R}}

\DeclarePairedDelimiter\abs{\lvert}{\rvert}%

\newcommand{\rank}{\operatorname{rank}}
\newcommand{\Aut}{\operatorname{Aut}}
\newcommand{\graph}{\operatorname{graph}}
\newcommand{\GS}{\mathcal{N}}
\newcommand{\sign}{\operatorname{sign}}
\newcommand{\HLS}{\Re_{++}}
\renewcommand{\Im}{\operatorname{Im}}

\newtheorem{definition}{Definition}
\newtheorem{lemma}[definition]{Lemma}
\newtheorem{proposition}[definition]{Proposition}

\newtheorem{corollary}[definition]{Corollary}

\newtheorem{theorem}[definition]{Theorem}
\newtheorem*{proposition*}{Proposition}
\theoremstyle{remark}

\title{The automorphism group and the non-self-duality of $p$-cones}
\author{
	Masaru Ito%
		\thanks{Department of Mathematics, College of Science and Technology, Nihon University,
			1-8-14 Kanda-Surugadai, Chiyoda-Ku, Tokyo 101-8308, Japan 
			(\texttt{ito.m@math.cst.nihon-u.ac.jp)}.}
	\and
	Bruno F. Louren\c{c}o%
	\thanks{Department of Mathematical Informatics, Graduate School of Information Science \& Technology, University of Tokyo, 7-3-1 Hongo, Bunkyo-ku, Tokyo 113-8656, Japan.
		(\texttt{lourenco@mist.i.u-tokyo.ac.jp})}
}

\begin{document}
	\maketitle
	
\begin{abstract}
In this paper,  we determine the automorphism group of the $p$-cones ($p\neq 2$) in dimension greater than two. 
In particular, we show that the automorphism group of those $p$-cones are the positive scalar multiples of the generalized permutation matrices that fix the main axis of the cone.
Next, we take a look at a problem related to the duality theory of the $p$-cones.
Under the Euclidean inner product it is well-known 
that a $p$-cone is self-dual only when $p=2$.
However, it was not known whether it is possible  to construct an inner product depending on $p$ which makes the $p$-cone self-dual.
Our results shows that no matter which inner product is considered, a
$p$-cone will never become self-dual unless $p=2$ or the dimension is less  than three. 
\end{abstract}

\section{Introduction}\label{sec:int}
In this work, we prove two results on the 
structure of the $p$-cones
$$
\SOC{p}{n+1}=\{(t,x)\in\Re\times \Re^n \mid t \geq \norm{x}_p\}.
$$
First, we describe the automorphism group of 
the $p$-cones $\SOC{p}{n+1}$ for $n \geq 2$ and $p \neq 2$, $1 < p < \infty$. We show that every automorphism of $\SOC{p}{n+1}$ must have the format
\begin{equation}\label{eq:format}
\alpha\begin{pmatrix}
1 & 0\\
0 & P
\end{pmatrix},
\end{equation}
where $\alpha > 0$ and $P$ is an $n\times n$ generalized permutation matrix.
The second result is that, for $n \geq 2$ and $p\neq 2$, it is not possible to construct an inner product on $\Re^{n+1}$ for which  $\SOC{p}{n+1}$ becomes self-dual.
In fact, the second result is derived as a corollary of a stronger result that $\SOC{p}{n+1}$ and $\SOC{q}{n+1}$ cannot be linearly isomorphic if $p<q$ and $n\geq 2$, except when $(p,q,n)=(1,\infty,2)$.

The motivation for this research is partly due to the work by Gowda and Trott \cite{GT14}, where they determined the automorphism group of $\SOC{1}{n+1}$ and $\SOC{\infty}{n+1}$.
However, they left open the problem of determining the automorphisms of the other $p$-cones, for $p \neq 2$.
Here, we recall that the case $p = 2$ correspond to the \emph{second order cones} and they are \emph{symmetric}, i.e., \emph{self-dual} and \emph{homogeneous}. The structure of second-order cones and their automorphisms follow from the more general theory of 
Jordan Algebras \cite{FK94}, see also \cite{LS75}. 

In \cite{GT14}, Gowda and Trott also proved that $\SOC{1}{n+1}$ and $\SOC{\infty}{n+1}$ are \emph{not} homogenous cones and they posed the problem of proving/disproving that $\SOC{p}{n+1}$ is not homogeneous for $p\neq 2$, $n \geq 2$.  Recall that a cone is said to be \emph{homogeneous} if its group of automorphisms acts transitively on the interior of the cone. 
In \cite{IL17}, using the theory of $T$-algebras \cite{V63}, we gave a proof that  $\SOC{p}{n+1}$ is not homogenous for $p\neq 2$, $n \geq 2$. However, there are two unsatisfactory aspects of our previous result.
The first is that we were not able to compute the automorphism group of $\SOC{p}{n+1}$. The second is 
that although we showed that $\SOC{p}{n+1}$ is not homogeneous, we were unable to obtain two elements $x,y$ in interior of $\SOC{p}{n+1}$ such that no automorphism 
of  $\SOC{p}{n+1}$ maps $x$ to $y$. 
That is, we were unable to show concretely how 
homogeneity breaks down on $\SOC{p}{n+1}$.
The results discussed here remedy those flaws and provide an alternative proof that $\SOC{p}{n+1}$ is not homogeneous.

Another motivation for this work is the general problem of determining when a closed convex cone $\stdCone \subseteq \Re^n$ is self-dual.
If $\Re^n$ is equipped with some inner product 
$\inProd{\cdot}{\cdot}$, the dual cone of $\stdCone$ is defined as 
$$
\stdCone^* = \{y \in \Re^n \mid \inProd{x}{y}\geq 0, \forall x \in \stdCone\}.
$$ 
As discussed in Section 1 of \cite{IL17}, an often overlooked point is that $\stdCone^*$ \emph{depends on $\inProd{\cdot}{\cdot}$}. 
Accordingly, it is entirely plausible that a cone that is not self-dual under the Euclidean inner product might become self-dual if the inner product is chosen appropriately.

This detail is quite important because  sometimes we see articles claiming that a certain cone is not a symmetric cone because it is not self-dual under the Euclidean inner product.
This is, of course, not enough. As long as a cone is homogeneous and there exists some inner product that makes it self-dual, the cone can be investigated under the theory of  Jordan Algebras.

This state of affairs brings us to the case of the $p$-cones. 
Up until  the recent  articles \cite{GT14,IL17},  there was no rigorous proof that the 
$p$-cones $\SOC{p}{n+1}$ were \emph{not} symmetric when $p \neq 2$ and $n \geq 2$.
Now, although we know that   $\SOC{p}{n+1}$ is not homogeneous for $p \neq 2$ and $n \geq 2$, it still remains to investigate whether $\SOC{p}{n+1}$ could become self-dual under an appropriate inner product.
This question was partly discussed by Miao, Lin and Chen in \cite{MLC17}, where they showed that a $p$-cone (again, $p\neq 2$, $n\geq 2$) is not self-dual under  an inner product induced by a diagonal matrix.
The results described here show, in particular, that no inner product can make $\SOC{p}{n+1}$ self-dual, for $p \neq 2$, $n\geq 2$.

We now explain some of the intuition behind our proof techniques.
Let $n\geq 2$ and let $f_p:\Re^n\setminus \{0\} \to \Re$ be the function that maps 
$x$ to $\norm{x}_p$. 
When $p \in (1,2)$, we have that $f_p$ is twice differentiable only at points $x$ for which 
$x_i \neq 0$, for all $i$. 
In contrast, if $p \in (2,\infty)$, $f_p$ is twice differentiable throughout $\Re^n\setminus \{0\}$.
Now, we let $M_p$ be  the boundary without the zero  of the cone $\SOC{p}{n+1}$.
With that, $M_p$ is exactly the graph of the function $f_p$. Furthermore, $M_p$ is a 
$C^1$-embedded smooth manifold if $p \in (1,2)$. If $p \in (2,\infty)$, $M_p$ is a $C^2$-embedded smooth manifold.
Now, any linear bijection between $\SOC{p}{n+1}$ and $\SOC{q}{n+1}$ must map the boundary of  $\SOC{p}{n+1}$ to the boundary of $\SOC{q}{n+1}$, thus producing a map between $M_p$ and $M_q$.
Then, if $p \in (1,2)$ and 
$q \in (2,\infty)$, there can be no linear bijection between  $\SOC{p}{n+1}$ and $\SOC{q}{n+1}$ because this would establish a diffeomorphism between submanifolds that are embedded with different levels of smoothness.

Now suppose that $p,q$ are both in $(1,2)$ and 
that there exists some linear bijection $A$ between $\SOC{p}{n+1}$ and $\SOC{q}{n+1}$. 
If $(f_p(x),x) \in M_p$ is such that $f_p$ is \emph{not} twice differentiable at $x$, then $A$ must map $(f_p(x),x)$ to a point $(f_q(y),y)$ for which $f_q$ is \emph{not} twice differentiable at $y$. This idea is made precise in Proposition \ref{prop:C^k}.
In particular, this fact imposes severe restrictions on how $\Aut(\SOC{p}{n+1})$ acts on $\SOC{p}{n+1}$ and this is the key observation necessary for showing that the matrices in $\Aut(\SOC{p}{n+1})$ can be written as in \eqref{eq:format}.

This work is divided as follows. In Section \ref{sec:prel} we present the notation used in this paper and review some facts about cones, self-duality and $p$-cones.
In Section \ref{sec:diff}, we discuss the tools from manifold theory necessary for our discussion. Finally, in Section \ref{sec:main} we prove our main results.


\section{Preliminaries}\label{sec:prel}
A \emph{convex cone} is a subset $\stdCone$ of some real vector space $\Re^n$ such that $\alpha x + \beta y \in \stdCone$ holds whenever $x,y \in \stdCone$ and $\alpha,\beta \geq 0$. A cone $\stdCone$ is said to be \emph{pointed} if $\stdCone \cap - \stdCone = \{0\}$.
For a subset $S$ of $\Re^n$, the (closed) \emph{conical hull} of $S$, denoted by $\conicHull(S)$, is the smallest closed convex cone in $\Re^n$ containing $S$. If 
$v \in \Re^n$, we write $\Re_+(v)$ for the half-line generated by $v$ and $\HLS$ for $\Re_+(v) \setminus \{0\}$, i.e., 
\begin{align*}
\Re_+(v)& = \{\alpha v \mid \alpha \geq 0\},\\
\HLS(v)& = \{\alpha v \mid \alpha > 0\}.
\end{align*}
A convex subset $\stdFace$ of $\stdCone$ is said to be a \emph{face} of $\stdCone$ if
the following condition hold:
If $x,y \in \stdCone$ satisfies $\alpha x + (1-\alpha)y \in \stdFace$ for some $\alpha \in (0,1)$ then $x,y \in \stdFace$ holds.
A one dimensional face is called an \emph{extreme ray}.
A \emph{polyhedral convex cone} is a convex cone that can be expressed as the solution set of finitely many linear inequalities.

If $\inProd{\cdot}{\cdot}$ is an inner product on $\Re^n$, we can define the \emph{dual cone} of $\stdCone$ with respect to the inner product $\inProd{\cdot}{\cdot}$ by
$$\stdCone^* = \{x \in \Re^n \mid \inProd{x}{y}\geq 0,~\forall y \in \stdCone\}.$$
A convex cone $\stdCone$ is \emph{self-dual} if there exists an inner product on $\Re^n$ for which the dual cone coincides with $\stdCone$ itself.

Two convex cones $\stdCone_1$ and $\stdCone_2$ in $\Re^n$ are said to be \emph{isomorphic} if there exists a linear bijection $A \in GL_n(\Re)$, called an \emph{isomorphism}, such that $A\stdCone_1 = \stdCone_2$.
An \emph{automorphism} of a convex cone $\stdCone$ in $\Re^n$ is a map $A \in GL_n(\Re)$ such that $A\stdCone = \stdCone$.
The group of all automorphisms of $\stdCone$ is written by $\Aut(\stdCone)$ and called the \emph{automorphism group of $\stdCone$}.

A convex cone $\stdCone$ is said to be \emph{homogeneous} if $\Aut(\stdCone)$ acts transitively on the interior of $\stdCone$, that is,
for every elements $x$ and $y$ of the interior of $\stdCone$, there exists $A \in \Aut(\stdCone)$ such that $y = Ax$.

\subsection{On self-duality}
Let $\stdCone \subseteq \Re^n$ be a closed convex cone.
As we emphasized in Section \ref{sec:int}, self-duality is a relative concept and depends on what inner product we are considering. Let $\inProd{\cdot}{\cdot}_E$ denote 
the Euclidean inner product  and consider the dual of $\stdCone$ with respect $\inProd{\cdot}{\cdot}_E$.
$$
\stdCone^* = \{y \in \Re^n \mid \inProd{x}{y}_E \geq 0, \forall x \in \stdCone \}.
$$
We have the following proposition.
\begin{proposition}\label{prop:selfdual}
Let $\stdCone \subseteq \Re^n$ be a closed convex cone and let $\stdCone^*$ be the dual of $\stdCone$ with the respect to the Euclidean inner product $\inProd{\cdot}{\cdot}_E$.
Then, there exists an inner product on $\Re^n$ that turns $\stdCone$ into a self-dual cone if and only if there exists a symmetric positive definite matrix $A$ such that 
$A\stdCone = \stdCone^*$. 
\end{proposition}
\begin{proof}
First, suppose that there exist some inner product $\inProd{\cdot}{\cdot}_\stdCone$  for which $\stdCone$ becomes self-dual. Then, there is a symmetric positive definite matrix $A$ such that 
$$
\inProd{x}{y}_\stdCone = \inProd{x}{Ay}_E,
$$
for all $x, y \in \Re^n$. In fact, $A_{ij} = \inProd{e_i}{e_j}_{\stdCone}$, where $e_i$ is the $i$-th standard unit vector in $\Re^n$. By assumption, we have 
\begin{align*}
\stdCone & = \{x \in \Re^n \mid \inProd{x}{Ay}_E \geq 0, \forall y \in \stdCone  \}\\
& = \{x \in \Re^n \mid \inProd{Ax}{y}_E \geq 0, \forall y \in \stdCone  \}\\
& = A^{-1}\{z \in \Re^n \mid \inProd{z}{y}_E \geq 0, \forall y \in \stdCone  \}\\
& = A^{-1}\stdCone^*.
\end{align*}
This shows that $A\stdCone = \stdCone^*$.

Reciprocally, if $A\stdCone = \stdCone^*$, we define the inner product $\inProd{\cdot}{\cdot}_{\stdCone}$ such that $$
\inProd{x}{y}_\stdCone \coloneqq \inProd{x}{Ay}_E,
$$
for all $x, y \in \Re^n$. Then, a straightforward calculation shows that the dual of $\stdCone$ with respect $\inProd{\cdot}{\cdot}_{\stdCone}$ is indeed $\stdCone$.
\end{proof}
Therefore, determining whether $\stdCone$ is self-dual for some inner product boils down to determining the existence of a positive definite linear isomorphism between cones, which is a difficult problem in general. 

\subsection{$p$-cones}\label{sec:pcon}
Here we present some basic facts on $p$-cones.
The \emph{$p$-cone} is the closed convex cone in $\Re^{n+1}$ defined by
$$
\SOC{p}{n+1} = \{(t,x)\in\Re\times \Re^n \mid t \geq \norm{x}_p\}
$$
where $\norm{x}_p$ is the $p$-norm on $\Re^n$:
$$\norm{x}_p=(|x_1|^p+\cdots + |x_n|^p)^{1/p} ~\text{ for }~ p \in [1,\infty) ~\text{ and  }~ \norm{x}_\infty=\max(|x_1|,\ldots,|x_n|).$$

The dual cone of the $p$-cone with respect to the Euclidean inner product is given by
$
(\SOC{p}{n+1})^* = \SOC{q}{n+1}
$
where $q$ is the conjugate of $p$, that is, $\frac{1}{p}+\frac{1}{q}=1$.
The cones $\SOC{1}{n+1}$ and $\SOC{\infty}{n+1}$ are polyhedral.
In fact, $\SOC{1}{n+1}$ has $2n$ extreme rays $$
\Re_{+}(1,\sigma e_i^n),\quad i=1,\ldots,n,\quad \sigma\in\{-1,1\},
$$
where $e_i^n$ denotes the $i$-th standard unit vector in $\Re^n$.
Moreover, $\SOC{\infty}{n+1}$ has $2^n$ extreme rays
$$\Re_+(1,\sigma_1,\ldots,\sigma_n),\quad \sigma_1,\ldots,\sigma_n \in \{-1,1\}.$$
The difference in the number of extreme rays shows that  $\SOC{1}{n+1}$ and $\SOC{\infty}{n+1}$ are not isomorphic if $n\geq 3$. However, for $n = 2$, they are indeed isomorphic as
\begin{equation}\label{eq:k1kinf}
A\SOC{1}{3}=\SOC{\infty}{3},\quad
A =
\left(
\begin{array}{ccc}
1 & 0 & 0 \\
0 & \sqrt{2}\cos(\pi/4) & -\sqrt{2}\sin(\pi/4) \\
0 & \sqrt{2}\sin(\pi/4) & \sqrt{2}\cos(\pi/4)
\end{array}
\right)
= \left(
\begin{array}{ccc}
1 & 0 & 0 \\
0 & 1 & -1 \\
0 & 1 & 1
\end{array}
\right).
\end{equation}
The second order cone $\SOC{2}{n+1}$ is known to be a \emph{symmetric cone}, that is, it is both self-dual and homogeneous, admitting a Jordan algebraic structure \cite{FK94}.
The automorphism group of the second order cone can be identified by the result of Loewy and Schneider \cite{LS75}: $A\SOC{2}{n+1}=\SOC{2}{n+1}$ or $A\SOC{2}{n+1} = -\SOC{2}{n+1}$ holds if and only if $A^TJ_{n+1}A=\mu J_{n+1}$ for some $\mu>0$ where $J_{n+1}={\rm diag}(1,-1,\ldots,-1)$.

Gowda and Trott determined the structure of the automorphism group of the $p$-cones in the case $p=1,\infty$:

\begin{proposition}[Gowda and Trott, Theorem 7 in  \cite{GT14}]
\label{prop:GT}
For $n \geq 2$, $A$ belongs to $\Aut(\SOC{1}{n+1})$ if and only if $A$ has the form
$$
A = \alpha \left(
\begin{array}{cc}
1 & 0\\
0 & P
\end{array}
\right),
$$
where $\alpha>0$ and $P$ is an $n \times n$ \emph{generalized permutation matrix}, that is, a permutation matrix multiplied by a diagonal matrix whose diagonal elements are $\pm 1$.
Moreover, $\Aut(\SOC{\infty}{n+1})=\Aut(\SOC{1}{n+1})$ holds.
\end{proposition}

In particular, Proposition \ref{prop:GT} yields the following consequences.
\begin{itemize}
\item $\SOC{1}{n+1}$ and $\SOC{\infty}{n+1}$ are not homogeneous for $n \geq 2$ because any $A \in \Aut(\SOC{1}{n+1})=\Aut(\SOC{\infty}{n+1})$ fixes the ``main axis'' $\Re_+(1,0,\ldots,0)$ of these cones.
\item $\SOC{1}{n+1}$ and $\SOC{\infty}{n+1}$ are never self-dual for $n \geq 2$. This is a known fact, but we will also obtain this result as a consequence of Corollary~\ref{cor:selfdual} where Proposition~\ref{prop:GT} will be helpful to prove the case $n=2$. 
At this point, we should remark that Barker and Foran proved in Theorem 3 of \cite{BF76} that a self-dual polyhedral cone in $\Re^3$ must have an odd number of extreme rays. Since $\SOC{1}{3}$ and  $\SOC{\infty}{3}$ have four extreme rays, Barker and Foran's result implies that they are never self-dual.
\end{itemize}

\section{Manifolds, tangent spaces and the Gauss map}\label{sec:diff}
In this subsection, we will provide a brief overview of the  tools we will use from manifold theory, more details can be seen in Lee's book \cite{Lee2012} or the initial chapters of do Carmo's book \cite{Carmo92}. 
First, we recall that a \emph{$n$-dimensional smooth manifold $M$} is a second countable Haussdorf topological space  equipped with a collection $\mathcal{A}$ of maps $\varphi: U \to \Re^n$ with the following properties.
\begin{enumerate}[$(i)$]
	\item each map $\varphi \in \mathcal{A}$ is such that $\varphi(U)$ is an open set of $\Re^n$. Furthermore, $\varphi$ is an homeomorphism between $U$ and $\varphi(U)$, i.e., $\varphi$ is a continuous bijection with continuous inverse.
	\item if $\varphi:U \to \Re^n, \psi:V \to\Re^n$ both belong to $\mathcal{A}$ and $U \cap V \neq \emptyset$, then $\psi \circ \varphi^{-1}: \varphi^{-1}(U\cap V) \to \psi(U\cap V)$ is a $C^{\infty}$ diffeomorphism, i.e., $\psi \circ \varphi^{-1}$ is a bijective function such that $\psi \circ \varphi^{-1}$ and $\varphi\circ \psi^{-1}$ have continuous derivatives of all orders.
	\item for every $x \in M$, we can find a map $\varphi \in \mathcal{A}$ for which $x$ belongs to the domain of $\varphi$.
	\item if $\psi$ is another map defined on a subset of $M$ satisfying $(i)$ and $(ii)$, then $\psi \in \mathcal{A}$. That is, $\mathcal{A}$ is maximal.
\end{enumerate}
The set $\mathcal{A}$ is called a \emph{maximal smooth atlas} and the  maps in $\mathcal{A}$ are called \emph{charts}. If $\varphi: U \to \Re^n$ is a chart and $x \in U$, we say that $\varphi$ is a \emph{chart around $x$}.

Let $M_1,M_2$ be smooth manifolds and $f:M_1\to M_2$ be a function. The function $f$ is said to be \emph{differentiable at $x \in M_1$} if there is a chart $\varphi$ of $M_1$ around $x$ and a chart $\psi$ of $M_2$ around $f(x)$ such that $$\psi\circ f \circ \varphi^{-1}$$ is differentiable at $\varphi(x)$. 
Then, $f$ is said to be \emph{differentiable}, if it is differentiable throughout $M_1$.
Similarly, we say that $f$ is differentiable of class $C^k$ if $\psi\circ f \circ \varphi^{-1}$ is of class $C^k$, for every pair of charts of $M_1$ and $M_2$ such that the image of $\varphi^{-1}$ and the domain of $\psi$ intersect.
Whether a function is differentiable at some point or is of class $C^k$ does not depend on the particular choice of charts.
The function  $\psi\circ f \circ \varphi^{-1}$ is also said to be a \emph{local representation of $f$}.
If $f$ is a bijection such that it is $C^k$ everywhere and whose inverse $f^{-1}$ is also $C^k$ everywhere, then $f$ is said to be a \emph{$C^k$ diffeomorphism}.

Let $M$ be a $n$-dimensional smooth manifold.
Let $C^{\infty}(M)$ denote the ring of $C^{\infty}$ real functions $g:M\to \Re$. A derivation of $M$ at $x$ is a function $v:C^{\infty}(M) \to \Re$ such that for every $g,h \in C^{\infty}(M)$, we have
$$
v(gh)= (v(g))h(x) + g(x)v(h).
$$
Given a $n$-dimensional smooth manifold $M$ and $x \in M$, we write $T_{x}M$ for the tangent space of $M$ at $x$, which is the subspace of derivations of $M$ at $x$. It is a basic fact that the dimension of $T_{x}M$ as a vector space coincides with  the dimension of $M$ as a smooth manifold. 

Let $f:M_1 \to M_2$ be a $C^1$ map between smooth manifolds. Then, at each $x \in M_1$, $f$ induces a linear map between $df_x:T_{x}M_1 \to T_{f(x)}M_2$ such that 
given $v \in T_{x}M_1$, $df_x(v)$ is the derivation of $M_2$ at $f(x)$ satisfying
$$
(df_x(v))(g)= v(g\circ f),
$$ 
for every $g \in C^{\infty}(N)$. The map $df_x$ is the \emph{differential map of $f$ at $x$}.
If the linear map $df_x$ is injective everywhere, then $f$ is said to be an \emph{immersion}.
Furthermore, if $f$ is a $C^k$ diffeomorphism with $k \geq 1$, then $df_x$ is a linear bijection for every $x$. Recall that in order to check whether $f$ is immersion, it is enough to check that the local representations of $f$ are immersions.

Now, suppose that $\alpha:(-\epsilon, \epsilon)\to M$ is a $C^{\infty}$ curve with $\alpha(0) = x$. 
Then $d\alpha _0(0) \in T_{x}M$.
Furthermore, $T_{x}M$ coincides with the set of velocity vectors of smooth curves passing through $x$. With  a slight abuse of notation, let us write $\alpha'(t) = d\alpha _0(t)$. With that, we have
\begin{align}\label{eq:vel_tan}
T_{x}M = \{ \alpha'(0) \mid \alpha:(-\epsilon,\epsilon)\to M, \alpha(0) = x, \alpha \text{ is } C^{1} \},
\end{align}
see more details in Proposition 3.23 and pages 68-71 in \cite{Lee2012}.
With this, we can compute a differential 
$df_x(v)$ by first selecting a $C^1$ curve $\alpha$ contained in $M$ with $\alpha(0) = x$, $\alpha'(0)=v$. Then, we have 
$df_x(v) = (f\circ \alpha)'(0)$, see Proposition 3.24 in \cite{Lee2012}.

A map $\iota:M_1\to M_2$ is said to be a \emph{$C^k$-embedding} if it is a $C^k$ immersion and a 
homeomorphism on its image (here, $\iota(M_1)$ has the subspace topology induced from $M_2$).
Now, suppose that, in fact,  $M_1 \subseteq M_2$ and let $\iota:M_1\to M_2$ denote the inclusion map, i.e., 
$\iota(x) = x$, for all $x \in M_1$.
If $\iota$ is a $C^k$ embedding, we say that 
$M_1$ is a \emph{$C^k$-embedded submanifold of $N$}.

We remark that when $M$ is a $m$-dimensional $C^k$-embedded submanifold of $\Re^n$, the requirement that 
$\iota$ be an a $C^k$ embedding has the following consequences. First, the topology of $M$ has to be the subspace topology of $\Re^n$, i.e., the open sets of $M$ are open sets of $\Re^n$ intersected with $M$. Now, let $\varphi: U \to \Re^m$ be a chart of $M$. Then, $\iota \circ \varphi^{-1} : \varphi(U) \to U$ is a $C^k$ diffeomorphism. That is, although $\varphi^{-1}$ is $C^\infty$ when saw as a map between $\varphi(U)$ and $M$, its class of differentiability might decrease\footnote{Here is an example of what can happen. Let $M$ be graph of the function $f(x) =\abs{x}$. $M$ is a differentiable manifold and to create  a maximal smooth atlas for $M$ we first start with a set $\mathcal{A}$ containing only the map $\varphi : M \to \Re$ that takes $(\abs{x},x)$ to $x$. At this point, conditions $(i),(ii),(iii)$ of the definition of atlas are satisfied.	Then, we add to $\mathcal{A}$ every map $\psi$ such that $\mathcal{A}\cup \{\psi\}$ still satisfies $(i),(ii),(iii)$. The resulting set must be a maximal smooth atlas. Following the definition of differentiability between manifolds, the  map $\varphi^{-1}$ is $C^\infty$ if we see it as a map between $\Re \to M$, since $\varphi \circ \varphi^{-1}(x) = x$. However, 
$\iota \circ \varphi^{-1}$ is not even a $C^1$ map, because $\abs{x}$ is not differentiable at $0$. \label{fnt:diff}  } when seen as a map between $U$ and $\Re^m$.
For embedded manifolds of $\Re^n$, as a matter of convention, we will always see the inverse of a chart $\varphi$ as a function whose codomain is $\Re^n$ and we will omit the embedding $\iota$.

Furthermore, whenever $M$ is a $C^k$-embedded submanifold of $\Re^n$, we will define tangent spaces in a more geometric way. Given $x \in M$, we will define 
$T_xM$ as the space of tangent vectors of $C^1$ curves that pass through $x$:
\begin{equation}\label{eq:txm}
T_xM = \{\alpha'(0)  \mid \alpha:(-\epsilon,\epsilon)\to \Re^n, \alpha(0) = x, \alpha \subseteq M, \alpha \text{ is } C^1 \},
\end{equation}
where $\alpha \subseteq M$ means that $\alpha(t) \in M$, for every $t \in (-\epsilon, \epsilon)$.
Here, since we have an ambient space, $\alpha'(0)$ is the derivative of $\alpha$ at $0$ in the usual sense.

Both definitions of tangent spaces presented so far are equivalent in the following sense. Let $\tilde T_xM$ denote the space of derivations of $M$ at $x$ and let $\iota:M\to \Re^n$ denote the inclusion map.
Then, $d\iota _x$ is a map between $\tilde T_xM$ and $T_x\Re^n$. Then, identifying $T_x\Re^n$ with 
$\Re^n$, 
it holds that $d\iota _x(\tilde T_xM) = T_xM$. In particular, $\tilde T_xM$ and $T_xM$ have the same dimension.

Finally, we recall that for smooth manifolds, the topological notion of \emph{connectedness} is equivalent to the notion of \emph{path-connectedness}, see Proposition 1.11 in \cite{Lee2012}. Therefore, a manifold $M$ is connected if and only if for every $x,y \in M$ there is a continuous curve $\alpha:[0,1]\to M$ such that $\alpha(0) = x$ and $\alpha(1) = y$.

\subsection{Graphs of differentiable maps}
For a real valued function $f:U\to \Re$ defined on $U \subseteq \Re^n$,
the graph of $f$ is defined by
\[
\graph f := \{(y,x) \in \Re\times U \mid y = f(x) \} \subseteq \Re^{n+1}.
\]
In item $(i)$ of the next proposition, for the sake of completeness, we give a proof of the well-known fact that if $f$ is a $C^k$ function, then $\graph f$ must be a $C^k$-embedded manifold. 
In item $(ii)$ we observe the fact, also known but perhaps less well-known, that the converse also holds.
This is important for us because if we know that $f$ is $C^1$ but not $C^2$, then this creates an obstruction to the existence of certain maps between 
$\graph f$ and $C^2$ manifolds.

\begin{proposition}\label{prop:graph}
For $k \geq 1$,
let $f:U\to \Re$ be a $C^1$ function defined on an open subset $U$ of $\Re^n$.
\begin{enumerate}
\item[(i)]
If $f$ is $C^k$ on an open subset $V$ of $U$,
then
$\graph f|_V$ is an $n$-dimensional $C^k$-embedded submanifold of $\Re^{n+1}$.
\item[(ii)] Suppose that a subset $M$ of $\graph f$ is an $n$-dimensional $C^k$-embedded submanifold of $\Re^{n+1}$, with $k \geq 1$.
Then $f$ is $C^k$ on the open set $\pi_U(M)$, where $\pi_U:\Re\times U \to U$ is the projection onto $U$.
\end{enumerate}
\end{proposition}
\begin{proof}	
$(i)$ The proof here is essentially the one contained Example 1.30 and Proposition 5.4 of \cite{Lee2012}, except that here we take into account the level of smoothness of the embedding. 

First, let $M=\graph(f|_V)$ and consider the subspace topology inherited from $\Re^{n+1}$ (again, see Examples 1.3 and 1.30 in \cite{Lee2012} for more details).
With the subspace topology, the map $\varphi:V\to M$, given by 
$$\varphi(x)=(f(x),x)$$ 
is a homeomorphism between $V$ and $M$, whose inverse is the projection restricted to $M$, that is $\varphi^{-1}(f(x),x)=x$.
Furthermore, $\varphi^{-1}$ induces a maximal smooth atlas of $M$ making $\varphi^{-1}:M\to V$ a chart\footnote{The idea is the same as in Footnote \ref{fnt:diff}, we start with $\mathcal{A} = \{\varphi^{-1}\}$ and add every map  $\psi$ for which $\mathcal{A}\cup \{\psi\}$ still satisfies properties $(i),(ii),(iii)$ of the definition of atlas. }.
We now check that the inclusion $\iota:M \to \Re^{n+1}$ is a $C^k$ embedding. 
A local representation for $\iota$ is obtained 
by considering $\iota \circ \varphi:V\to\Re^{n+1}$, which shows 
that $\iota$ is a $C^k$ differentiable map.
The  inverse $\iota^{-1}:\iota(M) \to M$ is given by restricting the identity map in $\Re^{n+1}$ to $M$. 
Since the topology on $M$ is the subspace topology, this establishes that $\iota$ is an homeomorphism.

Furthermore, since the $(n+1)\times n$ Jacobian matrix $J_{\iota \circ \varphi}$ of the representation of $\iota$ has rank $n$, we see that $\iota$ is an immersion. Hence, $M$ is a $C^k$-embedded submanifold of $\Re^{n+1}$.

$(ii)$
Take $x_0 \in \pi_U(M)$.
Let $\Phi:V \to \Re^n$ be a chart of $M$ around $(f(x_0),x_0)$.
We can write the map $\Phi^{-1}$ as
$$\Phi^{-1}(z)=(\psi(z),\varphi(z)) \in \Re\times U ~\text{ for }~ z \in \Phi(V),$$
for functions $\psi: \Phi(V) \to \Re$, $\varphi: \Phi(V) \to U$.
Since $\Im\Phi^{-1} \subseteq M \subseteq \graph f$, we have $\psi(z)=f(\varphi(z))$ for all $z \in \Phi(V)$.
Then we obtain a local representation $\tilde{\iota}:\Phi(V)\subseteq \Re^n\to \Re^{n+1}$ of the inclusion map $\iota:M\to\Re^{n+1}$ as follows:
\[
\tilde{\iota}(z) := \iota\circ\Phi^{-1} = (\psi(z),\varphi(z)) = (f \circ\varphi(z),\varphi(z)).
\]
Since $M$ is $C^k$-embedded, $\varphi$ and $\psi$ are $C^k$ when seen as maps $\Phi(V)\to \Re$ and $\Phi(V)\to \Re^n$, respectively.
Let $z_0=\Phi((f(x_0),x_0))$. Then $\varphi(z_0)=x_0$ since $$(f(x_0),x_0)=\Phi^{-1}(z_0)=(\psi(z_0),\varphi(z_0)).$$
Note that $\rank(J_{\tilde{\iota}}(z_0))=n$ holds because $\iota$ is an immersion.
On the other hand, since $f$ is $C^1$ by the assumption, it follows by the chain rule for the function $\psi=f\circ \varphi$ that
\[
J_{\psi}(z_0) = J_{f}(\varphi(z_0))J_{\varphi}(z_0) = J_{f}(x_0)J_{\varphi}(z_0).
\]
This means that each row of $J_{\psi}(z_0)$ is a linear combination of rows of $J_{\varphi}(z_0)$.
Therefore, we conclude that
\[ n = \rank J_{\tilde\iota}(z_0) = \rank (J_{\psi}(z_0)^T,J_{\varphi}(z_0)^T)^T = \rank J_{\varphi}(z_0).\]
Namely, the $n \times n$ matrix $J_{\varphi}(z_0)$ is nonsingular.
Since ${\varphi}$ is $C^{k}$, the inverse function theorem states that there exists a $C^{k}$ inverse ${\varphi}^{-1}:W\to \Re^n$ defined on a neighborhood $W$ of ${\varphi}(z_0)=x_0$.
Then, we conclude that the function
$$\psi\circ \varphi^{-1} = f\circ \varphi\circ \varphi^{-1} = f$$
is $C^k$ on $W$.

To conclude, we will show that $\pi_U(M)$ is open.
Since $\varphi^{-1}(W)$ is contained in the domain $\Phi(V)$ of the map $\varphi$,
it follows that $W=\varphi\circ\varphi^{-1}(W)\subseteq \varphi(\Phi(V))$.
Now, let $z \in \Phi(V)$. By definition, we have
$$(\psi(z),\varphi(z))=\Phi^{-1}(z)\in V,$$
which shows that $\varphi(z) \in \pi _U(V)$.
Therefore, $\varphi(\Phi(V))\subseteq \pi_U(V)\subseteq \pi_U(M)$. Hence, we have $W\subseteq \pi_U(M)$ and so $\pi_U(M)$ is open in $\Re^n$, since $x_0$ was arbitrary.
\end{proof}

Given a diffeomorphism $A$ between two graphs of $C^1$ maps $f,g:U\to\Re$,
the next proposition shows a relation of the categories of differentiability of $f$ and $g$ through the diffeomorphism $B:U\to U$ defined by
$$
B(x)=\pi_U(A(f(x),x))
$$
where $\pi_U:\Re\times U\to U$ is the projection onto $U$.
The map $B$ will play a key role in the proof of our main result applied with $U=\Re^n\setminus\{0\}$, $f(x)=\norm{x}_p$ and $g(x)=\norm{x}_q$.
We give an illustration of the map $B$ in Figure~\ref{fig:B}.

\begin{figure}
	\centering
\includegraphics[scale=0.6]{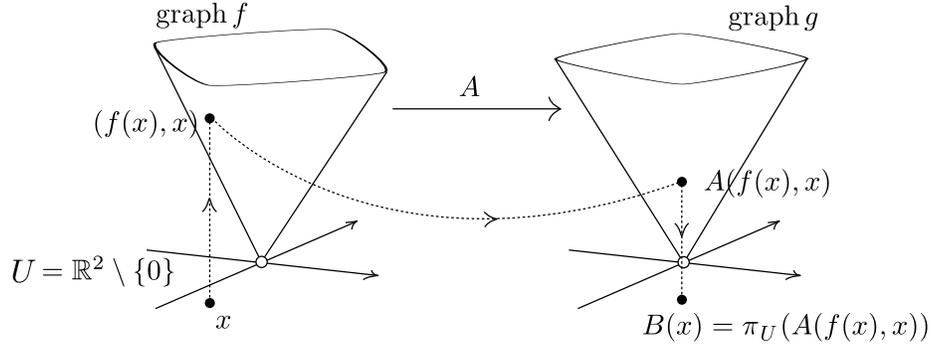}
\caption{Illustration of the map $B(x)=\pi_U(A(f(x),x))$}\label{fig:B}
\end{figure}

\newcommand{\mapA}{A}
\newcommand{\mapB}{B}

\begin{proposition}\label{prop:C^k}
	Let $f,g:U\to \Re$ be $C^1$ maps defined on an open subset $U$ of $\Re^n$.
	Suppose that $\mapA:\Re^{n+1}\to\Re^{n+1}$ is a $C^\infty$ diffeomorphism such that $\mapA(\graph f)=\graph g$.
	\begin{enumerate}
		\item[(i)]
		The map $\mapB:U\to U$, $\mapB(x):=\pi_U(\mapA(f(x),x))$ is a $C^1$ diffeomorphism, where $\pi_U:\Re\times U\to U$ satisfies $\pi_U(y,x)=x$.
		\item[(ii)]
		For $k \geq 1$, $f$ is $C^k$ on a neighborhood of $x$ if and only if $g$ is $C^k$ on a neighborhood of $\mapB(x)$.
	\end{enumerate}
\end{proposition}
\begin{proof}
	$(i)$
	Since $f$ is $C^1$ while $\pi_U$ and $\mapA$ are  $C^\infty$ maps, it is must be the case that $\mapB(x)=\pi_U(\mapA(f(x),x))$ is $C^1$.
	
	Let us check that the inverse of  $\mapB$ is the map $\mapB^{-1}(y)=\pi_U(\mapA^{-1}(g(y),y))$.
	Denote $$\mapB'(y)=\pi_U(\mapA^{-1}(g(y),y)).$$ For any $x \in U$, the relation $A(\graph f)=\graph g$ implies the existence of $y \in U$ such that $\mapA(f(x),x)=(g(y),y)$. Then we have
	$$
	\mapB(x)=\pi_U(\mapA(f(x),x))=\pi_U(g(y),y)=y.
	$$
	and, therefore,
	$$
	\mapB'(\mapB(x))=\mapB'(y)=\pi_U(\mapA^{-1}(g(y),y))=\pi_U(f(x),x)=x.
	$$
	Similarly, we obtain $\mapB(\mapB'(y))=y$. Hence, $\mapB^{-1}(y)=\mapB'(y)$ holds.
	
	Since $\mapB^{-1}(y)=\pi_U(\mapA^{-1}(g(y),y))$ is also $C^1$, we conclude that $\mapB$ is a $C^1$ diffeomorphism.
	
	$(ii)$
	If $f$ is $C^k$ on a neighborhood $V$ of $x$, then $\graph(f|_V)$ is an $n$-dimensional $C^k$-embedded submanifold of $\Re^{n+1}$ by Proposition~\ref{prop:graph} $(i)$.
	Then, by the assumption on $\mapA$, the set $M:=\mapA(\graph f|_V)$ is also an $n$-dimensional $C^k$-embedded submanifold of $\Re^{n+1}$ which satisfies $M\subseteq \graph g$.
	Therefore Proposition~\ref{prop:graph} $(ii)$ implies that $g$ is $C^k$ on the open set $\pi_U(M)=\pi_U(\mapA(\graph f|_V))$ which contains the point $\pi_U(\mapA(f(x),x))$.
	
	The converse of the assertion follows by applying the same argument to the diffeomorphism $\mapA^{-1}$ because $\mapA^{-1}(\graph g)=\graph f$ and $\pi_U(\mapA^{-1}(g(y),y))=x$ holds for $y=B(x) = \pi_U(\mapA(f(x),x))$.
\end{proof}

\subsection{The Gauss map}\label{sec:gauss}
In this subsection, let $M$ be a $C^k$-embedded submanifold of $\Re^n$ with dimension $n-1$ and $k \geq 1$. 
In this case, $M$ is sometimes called a \emph{hypersurface}  and when $n = 3$, $M$ is called a \emph{surface}. The differential geometry of surfaces is, of course, a classical subject discussed in many books, e.g., \cite{Carmo1976}.

In the theory of surfaces, a \emph{Gauss map} is a continuous function that associates to $x \in M$ a unit vector which is orthogonal to $T_xM$.
Unless $M$ is an orientable surface, it is not possible to construct a Gauss map that is defined globally over $M$. However, given any $x \in M$, it is always possible to construct a Gauss map in a neighborhood of $x$.
For the sake of self-containment, we will give a brief account of the construction of the Gauss map for hypersurfaces.

For what follows, we suppose that $\Re^n$ is equipped with some inner product $\inProd{\cdot}{\cdot}$ and the norm is given by
 $\norm{x} = \sqrt{\inProd{x}{x}}$, for all $x \in \Re^n$. 
Recalling \eqref{eq:txm}, $T_xM$ is seen as a subspace of $\Re^n$ and we will equip $T_xM$ with the same inner product $\inProd{\cdot}{\cdot}$.

\begin{definition}
	Let $M$ be a $C^k$-embedded submanifold of $\Re^n$ and let $x \in M$.  A $C^r$ Gauss map around $x$ is a $C^r$ function $\GS: U \to \Re^n$ such that $U\subseteq M$ is a neighborhood of $x$ in $M$ and 
	$$
	\GS(x) \in (T_xM)^\perp \quad \text{ and }\quad \norm{\GS(x)} = 1,
	$$
	for all $x \in U$, where $(T_xM)^\perp$ is the orthogonal complement to $T_xM$.
\end{definition}
For what follows, let $x^1,\ldots, x^n \in \Re^n$ and let 
$\det(x^1,\ldots,x^n)$ denote the determinant of the matrix such that its $i$-th column is given by $x^i$. 
Since the determinant is a multilinear function, if we fix the first $n -1 $ elements, we obtain a linear functional $f$ such that 
$$
f(x) = \det(x^1,\ldots,x^{n-1},x).
$$
Since $f$ is a linear functional, there is a unique vector $\Lambda(x^1,\ldots,x^{n-1}) \in \Re^n$ satisfying 
$$\inProd{\Lambda(x^1,\ldots,x^{n-1}) }{x} = f(x),
$$ 
for all $x \in \Re^n$. 
Furthermore, $\Lambda(x^1,\ldots,x^{n-1}) = 0$ is zero if and only if the $x^i$ are linearly dependent.

\begin{proposition}\label{prop:gauss}
	Let $M\subseteq \Re^n$ be an $(n-1)$ dimensional $C^k$-embedded manifold, with $k \geq 1$. Then, for every chart $\varphi:U\to \Re^{n-1}$,  
	there exists a $C^{k-1}$ local Gauss map of $M$  defined over $U$.
\end{proposition}
\begin{proof}
	Let $\varphi:U \to \Re^{n-1}$ be a chart of $M$.
	Then, $\varphi^{-1}$ is a function with domain $\varphi(U)$ (which is an open set of $\Re^{n-1}$) and codomain $\Re^n$.
	Let $u \in U$. It is well-known that the partial derivatives of $\varphi^{-1}$ at $\varphi(u)$ are a basis for $T_uM$, e.g., page 60 and Proposition 3.15 in \cite{Lee2012}.
	Let $v^i(u)$ be the partial derivative of $\varphi^{-1}$ at $\varphi(u)$ with respect the $i$-th variable.
	We define a Gauss map $N$ over $U$ by letting
	$$
	N(x) = \frac{\Lambda(v^1(u),\ldots,v^{n-1}(u))}{\norm{\Lambda(v^1(u),\ldots,v^{n-1}(u))}}.
	$$
	Since the $v^i(u)$ are a basis for $T_uM$, $\Lambda(v^1(u),\ldots,v^{n-1}(u))$ is never zero. 
	In addition, because  $\varphi^{-1}$ is of class $C^k$, 
	$N$ must be of class $C^{k-1}$. 
	
\end{proof}

\subsection{A lemma on hyperplanes and embedded submanifolds}
Let $M$ be a connected $C^1$-embedded $n-1$ dimensional submanifold of $\Re^n$ (i.e., a hypersurface) that is contained in a finite union of distinct hyperplanes $H_1,\ldots, H_r$.
The goal of this section is to prove that $M$ must be entirely contained in one of the hyperplanes.
The intuition comes from the case 
$n = 3$: a surface in $\Re^3$ cannot, say, be contained in $H_1\cup H_2$ and also intersect both $H_1$ and $H_2$ because it would generate a ``corner'' at the intersection $M\cap H_1 \cap H_2$, thus destroying smoothness. This is illustrated in Figure \ref{fig:hyp}.

This is probably a well-known differential geometric fact but we could not find a precise reference, so we give a proof here.
Nevertheless, our discussion is related to the following classical fact: 
a point in a surface for which the derivative of the Gauss map vanishes is called a \emph{planar point} and 
a connected surface in $\Re^3$ such that all its points are planar must be a piece of a plane, see Definitions 7, 8 and the proof of Proposition 4 of  Chapter 3 of \cite{Carmo1976}.

In our case, the fact that $M$ is contained in a finite number of hyperplanes hints that the image of any  Gauss map of  $M$ should be confined to the directions that are orthogonal to those hyperplanes. This, by its turn, suggests that the derivative of $N$ should vanish everywhere, i.e., all points must be planar.
In fact, our proof is inspired by the proof of  Proposition 4 of  Chapter 3 of \cite{Carmo1976} and we will use the same compactness argument at the end.

\begin{figure}
	\centering
\includegraphics[scale=0.3]{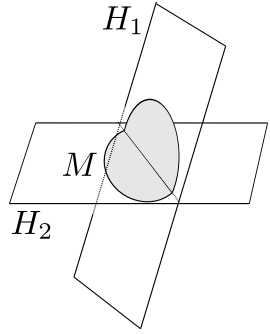}
\caption{A surface $M$ cannot be smooth if it is connected, contained in $H_1 \cup H_2$, but not entirely contained in neither $H_1$ nor $H_2$.}\label{fig:hyp}
\end{figure}

To start, we observe that the tangent of a curve contained in $H_1,\ldots, H_r$ must also be contained in those hyperplanes.
\begin{proposition}\label{prop:curve}
Let $H_i = \{a_i\}^\perp$ be hyperplanes  in $\Re^n$ for 
$i = 1, \ldots, r$.
Suppose that a $C^1$ curve $\alpha:(-\epsilon, \epsilon) \to \Re^n$ is contained in $X = \bigcup _{i=1}^r H_i$.
Then, $\alpha '(0) \in X$.
\end{proposition}
\begin{proof}
Changing the order of the hyperplanes if necessary, we may assume that 
\begin{align*}
\alpha(0) \in H_1\cap \cdots \cap H_s \\
\alpha(0) \not \in H_{s+1}, \ldots, H_r.
\end{align*}	
Since $\alpha$ is contained in $X$, we have $s \geq 1$.
Furthermore, because $\alpha$ is continuous, there is 
$\hat \epsilon > 0$ such that 
\begin{equation}\label{eq:aux1}
\alpha(\epsilon) \not \in H_{s+1}, \ldots,  H_r,
\end{equation}	
for $-\hat \epsilon < \epsilon < \hat \epsilon$.

Now, suppose for the sake of obtaining a contradiction 
that $\alpha'(0)$ does not belong to any of these hyperplanes 
$H_1, \ldots, H_s$. Therefore, for all $i \in \{1,\ldots, s\}$, we have 
$$
\inProd{\alpha(0)}{a_i} = 0, \qquad \inProd{\alpha'(0)}{a_i} \neq 0.
$$
Since $\alpha'(\cdot)$ is continuous, we can select $0 < \tilde \epsilon < \hat \epsilon$ such that  for all $i \in \{1,\ldots, s\}$ and $\epsilon \in (-\tilde \epsilon, \tilde \epsilon)$, we have
$$
\inProd{\alpha'(\epsilon)}{a_i} \neq 0.
$$
By the mean value theorem applied to $\inProd{\alpha(\cdot)}{a_i}$ on the interval $[0,\tilde{\epsilon}/2]$, we obtain that $\inProd{\alpha (\tilde \epsilon / 2)}{a_i} \neq 0$, for all $i \in {1,\ldots, s}$. Since $\tilde \epsilon /2 \in (-\hat \epsilon, \hat \epsilon)$, \eqref{eq:aux1} implies that 
$$\inProd{\alpha(\tilde \epsilon / 2)}{a_i} \neq 0,$$ for $i \in \{s+1,\ldots, r\}$ too. This shows that 
$\alpha(\tilde \epsilon / 2) \not \in X$, which is a contradiction.
\end{proof}
Before we prove the main lemma of this subsection, we need the following observation on finite dimensional vector spaces.
\begin{proposition}\label{prop:sub}
A finite dimensional real vector space $V$ is not a countable union of subspaces of dimension strictly smaller than $\dim V$.
\end{proposition}
\begin{proof}
Suppose that  $V$ is a countable union $\bigcup W_i$ of subspaces of dimension smaller than $\dim V$. Take the unit ball $B \subseteq V$. Then, $B = \bigcup W_i \cap B$. However, this is not possible since each $W_i \cap B$ has measure zero, while $B$ has nonzero measure.
\end{proof}
We now have all the necessary pieces to prove the main lemma.
\begin{lemma}\label{prop:hplane}
Let $X \subseteq \Re^n$ be a  union of finitely many hyperplanes $H_i = \{a_i\}^\perp$ , $a_i \neq 0$, $i = 1,\ldots, r$.	
Let $M$ be an $(n-1)$ dimensional differentiable manifold that is connected, $C^1$-embedded in $\Re^n$ and contained in 
$X$.
Then, $M$ must be entirely contained in one of the $H_i$.
\end{lemma}
\begin{proof}
We proceed by induction in $r$. 
The case $r = 1$ is clear, so suppose that $r > 1$.	
	
Consider a chart $\varphi:U \to \Re^{n-1}$ such that $U\subseteq M$ is connected and construct a $C^0$ (i.e., continuous) Gauss map $N$ in $U$, as in Proposition \ref{prop:gauss}.
Let $u \in U$ and let us examine the  tangent 
space $T_uM$. We have
$$
T_uM = \{\alpha'(0) \mid \alpha:(-\epsilon,\epsilon)\to M,~ \alpha(0) = u,~  \alpha \text{ is } C^1 \}.
$$	
By Proposition \ref{prop:curve},
$$
T_uM \subseteq X.
$$
Therefore, 
$$
T_uM = \bigcup _{i=1}^r H_i \cap T_uM.
$$
Each $H_i \cap T_uM$ is a subspace of $T_uM$ (an intersection of subspaces is also a subspace!). 
By Proposition~\ref{prop:sub}, $T_uM$ cannot be a union of subspaces of dimension less than $\dim T_uM = n-1$.
Therefore, there exists some index $j$ such that 
$H_j \cap T_uM = T_uM$.  Since both $T_uM$ and $H_j$ have dimension $n-1$, we conclude that $H_j = T_uM$.

In particular, the Gauss map $N$ satisfies $N(u) = a_j/\norm{a_j}$ or $N(u) = -a_j/\norm{a_j}$.
Therefore, for all $u \in U$, we have 
$$
N(u) \in \left\{\pm \frac{a_i}{\norm{a_i}} \mid i=1,\ldots, r\right\}.
$$
Since $U$ is connected and $N$ is continuous, we conclude that the Gauss map $N$ is constant. Denote this constant vector by $v$.

Let $\psi = \inProd{\varphi^{-1}(\cdot)}{v}$. Since 
$\varphi$ is a chart, given any $w\in \varphi(U)$, the differential $$d\varphi^{-1}_w:\Re^{n-1}\to T_{\varphi^{-1}(w)}M$$  is a linear bijection. Since $T_{\varphi^{-1}(w)}M$ is orthogonal to $v$, we conclude that $\psi' = 0$. Therefore $\psi$ must be constant and
there is $\kappa_0$ such that  
$\inProd{\varphi^{-1}(w)}{v} = \kappa_0$, for all $w \in \varphi(U)$. That is,  $\inProd{u}{v} = \kappa_0$, for all $u \in U$.

Recall that, given $x \in M$, we can always obtain a chart $\varphi:U \to M$ around $x$ such that $U$ is connected.
Therefore, the discussion so far shows that 
every $x \in M$ has a neighborhood $U$ such that $U$ is entirely contained in a hyperplane $$\{z \mid \inProd{z}{v_x}= \kappa_x\},$$ 
where $v_x$ has the same direction as one of the $a_1,\ldots, a_r$. Now, fix some $x \in M$ and let $y \in M$, $y \neq x$. Since $M$ is connected, there is a continuous path $\alpha:[0,1] \to M$ such that $\alpha(0) = x$ and 
$\alpha(1) = y$.

Similarly, for every $t \in [0,1]$, we can find a neighborhood $U_t\subseteq M$ of 
$\alpha(t)$ such that $U_t$ is contained in a hyperplane 
$\{z \mid \inProd{z}{v_t} = \kappa_{t} \}$ where $v_t$ is parallel to one of $a_1,\ldots,a_r$. In particular 
$$
[0,1] \subseteq \bigcup_{t \in [0,1]} \alpha^{-1}(U_t).
$$
Since the $U_t$ are open in $M$ and $\alpha$ is continuous, the  $\alpha^{-1}(U_t)$ form an open cover for the 
compact set $[0,1]$. 
Therefore, the Heine-Borel theorem implies that a finite number of the $\alpha^{-1}(U_t)$ are enough to cover $[0,1]$.
As a consequence, $\alpha$ itself is contained in finitely many neighborhoods $U_{t_1},\ldots U_{t_\ell}$. 
Now, we note the following:
\begin{itemize}
	\item If $U_{t_i} \cap U_{t_j} \neq \emptyset$ then $U_{t_i} \cap U_{t_j}$ is a nonempty open set in $M$ and therefore, an embedded submanifold of dimension $n-1$, see Proposition 5.1 in \cite{Lee2012}. Furthermore $U_{t_i} \cap U_{t_j}$ is contained in 
	the set $$
	H = \{z \in \Re^n \mid \inProd{z}{v_{t_i}} = \kappa _{t_i},~\inProd{z}{v_{t_j}} = \kappa _{t_j}  \}.
	$$
	Therefore, the smooth manifold $H$ must have at least dimension $n-1$. We conclude that ``$\inProd{z}{v_{t_i}} = \kappa _{t_i}$'' and 
	``$\inProd{z}{v_{t_j}} = \kappa _{t_j}$'' define the same hyperplane. So, $U_{t_i}$ and $U_{t_j}$ are in fact, contained in the same hyperplane.
	\item 
 $U_{t_1}$ must intersect 
	some of the $U_{t_2},\ldots, U_{t_\ell}$ because if it does not, then $\alpha^{-1}(U_{t_1})$ and $\alpha^{-1}(\cup _{i=2}^n U_{t_i})$ disconnect the connected set $[0,1]$.
	Changing the order of the sets if necessary, we may therefore assume that $U_{t_1}$ and $U_{t_2}$ intersect and, therefore,  lie in the same hyperplane. Similarly, the union $ U_{t_1} \cup U_{t_2}$ must intersect one of the remaining neighborhoods $U_{t_3},\ldots, U_{t_\ell}$, lest we disconnect the interval $[0,1]$. 
	By  induction, we conclude that all neighborhoods lie in the same hyperplane. 
\end{itemize}
In particular, $x$ and $y$ lie in the same hyperplane and, therefore, $M$ is entirely contained in some hyperplane whose normal direction has the same direction as one of the $a_1,\ldots, a_r$.

So far, we have shown that $M$ is entirely contained in a
hyperplane of the form $$\{z \in \Re^n \mid \inProd{z}{v} = \kappa _0 \}.$$ 
Without loss of generality, we may assume that $v$ has the same direction as $a_1$.
If $\kappa _0 = 0$, we are done. Otherwise,
since $v$ has the same direction as $a_1$, it follows that
$M$ does not intersect $H_1$ and  
$$
M \subseteq \bigcup _{i=2}^r H_i.
$$
By the induction hypothesis, $M$ must be contained in 
one of the $H_2,\ldots, H_r$.
\end{proof}

\section{Main results}\label{sec:main}

In this section, we show the main results on $p$-cones.
We begin by observing a basic fact on the differentiability of $p$-norms.

\begin{lemma}\label{lem:pnorm}
Let $n \geq 2$ and $p \in (1,\infty)$.
\begin{enumerate}
\item[(i)]
$\norm{\cdot}_p$ is $C^1$ on $\Re^n\setminus\{0\}$.
\item[(ii)]
If $p \in (1,2)$ then $\norm{\cdot}_p$ is $C^2$ on a neighborhood of $x$ if and only if $x_i\ne 0$ for all $i$.
\item[(iii)]
If $p \in [2,\infty)$ then $\norm{\cdot}_p$ is $C^2$ on $\Re^n\setminus\{0\}$.
\end{enumerate}
\end{lemma}
\begin{proof}
$(i)$ $\norm{\cdot}_p$ is $C^1$ on $\Re^n\setminus\{0\}$ because
$$\frac{\partial \norm{\cdot}_p}{\partial x_i}(x) = \norm{x}_p^{1-p}|x_i|^{p-1}\sign(x_i).$$

$(ii)$
If $x_i \ne 0$ for all $i$, it is easy to see that $\norm{\cdot}_p$ is $C^2$ on a neighborhood of $x$. 
For the converse, consider a point $x \ne 0$ with $x_i=0$ for some $i$. Then, $\frac{\partial \norm{\cdot}_p}{\partial x_i}(x)=0$ holds and so
\begin{align*}
\lim_{h\to 0}\frac{1}{h}\left(\frac{\partial \norm{\cdot}_p}{\partial x_i}(x+he_j)-\frac{\partial \norm{\cdot}_p}{\partial x_i}(x) \right)
&
=\lim_{h\to 0} h^{-1}\frac{\partial \norm{\cdot}_p}{\partial x_i}(x+he_j)
\\
&
= \lim_{h\to 0} h^{-1}\norm{x+he_i}_p^{1-p}|h|^{p-2}h
\\
&
= \lim_{h\to 0} \norm{x+he_j}_p^{1-p}|h|^{p-2}
\\
&
= \left\{\begin{array}{ll} +\infty& (p<2)\\  0 &(p>2)\end{array}\right.
.
\end{align*}
Hence, when $p \in (1,2)$, the derivative $\frac{\partial^2 \norm{\cdot}_p}{\partial x_j \partial x_i}(x)$ exists if and only if $x_i \ne 0$.

$(iii)$
For $p>2$ (the assertion in the case $p=2$ is clear),
$$\frac{\partial^2 \norm{\cdot}_p}{\partial x_j \partial x_i}(x) = (1-p)\norm{x}_p^{1-2p}|x_ix_j|^{p-1}\sign(x_i)\sign(x_j)$$ 
holds if $i \ne j$, otherwise we have
$$\frac{\partial^2 \norm{\cdot}_p}{\partial x_i^2}(x) = (1-p)\norm{x}_p^{1-2p}x_i^{2(p-1)}+(p-1)\norm{x}_p^{1-p}|x_i|^{p-2}.$$
\end{proof}
We now move on to the main result of this paper.
\begin{theorem}\label{thm:main}
Let $p,q \in [1,\infty]$, $p\leq q$, $n \geq 2$ and $(p,q,n)\neq (1,\infty,2)$.
Suppose that
$\SOC{p}{n+1}$ and $\SOC{q}{n+1}$ are isomorphic, that is,
$$A \SOC{p}{n+1} = \SOC{q}{n+1}$$ holds for some $A \in GL_{n+1}(\Re)$.
Then $p=q$ must hold. Moreover, if $p\ne 2$, then we have $A \in \Aut(\SOC{1}{n+1})$.
\end{theorem}
\begin{proof}
The proof consists of three parts I, II, and III.

\fbox{I}
First we consider the case $p \in \{1,\infty\}$ corresponding to the case when $\SOC{p}{n+1}$ is polyhedral.
Since $A$ preserves polyhedrality,  $q$ must be $1$ or $\infty$ too.
Note that $\SOC{1}{n+1}$ and $\SOC{\infty}{n+1}$ cannnot be isomorphic if $n\geq 3$ because they have different numbers of extreme rays, see Section \ref{sec:pcon}.
Therefore, $p=q=1$ or $p=q=\infty$ must hold.
Since $\Aut(\SOC{\infty}{n+1})=\Aut(\SOC{1}{n+1})$ holds (Proposition~\ref{prop:GT}), the assertion is verified in the case $p\in\{1,\infty\}$.

\fbox{II}~
Now let $p,q \in (1,\infty)$.
Then the set
$$M_p := \{(t,x)\in\Re\times\Re^n\setminus\{0\} \mid t = \norm{x}_p\}$$
becomes a $C^1$-embedded submanifold of $\Re^{n+1}$ by Lemma~\ref{lem:pnorm} (i) and Proposition~\ref{prop:graph} (i).
Note that $A\SOC{p}{n+1}=\SOC{q}{n+1}$ implies $AM_p=M_q$ since $A$ maps the boundary of $\SOC{p}{n+1}$ onto the boundary of $\SOC{q}{n+1}$.

It suffices to consider the case $p,q \in (1,2)$ by the following observation.
\begin{itemize}
\item[(a)]
The case $1<p<2\leq q<\infty$ does not happen in view of Proposition~\ref{prop:C^k} and Lemma \ref{lem:pnorm}. In fact, since $\norm{\cdot}_q$ is $C^2$ on $\Re^n\setminus\{0\}$ and $A^{-1}M_q = M_p$ holds, Proposition~\ref{prop:C^k} implies that $\norm{\cdot}_p$ is $C^2$ on $\Re^n\setminus\{0\}$ but this is a contradiction.
\item[(b)]
If $2\leq p\leq q<\infty$ holds, then taking the dual of the relation $A\SOC{p}{n+1}=\SOC{q}{n+1}$ with respect to the Euclidean inner product, it follows that
$$A^{-T}\SOC{p^*}{n+1} = \SOC{q^*}{n+1}$$
where $p^*$ and $q^* \in (1,2]$ are the conjugates of $p$ and $q$, respectively.
Either $p^*=q^*=2$ or $p^*,q^* \in (1,2)$ must hold by (a).
If $p^*=q^*=2$, then we are done since this implies that $p = q = 2$.
Now, suppose that $p^*,q^* \in (1,2)$. 
If we prove that $p^*=q^*$ and $A^{-T} \in \Aut(\SOC{1}{n+1})$, then we conclude that $p=q$ and $A \in \Aut(\SOC{1}{n+1})^{-T}$.
However, by  Proposition~\ref{prop:GT}, $\Aut(\SOC{1}{n+1})^{-T} = \Aut(\SOC{1}{n+1})$ (Note that, if $P$ is a generalized permutation matrix, then so is $P^{-T}$).
\end{itemize}
From cases $(a),(b)$ we conclude that it is enough to consider the case $p,q \in (1,2)$, which we will do next.

\fbox{III}~
Let $p,q \in (1,2)$.
We show by induction on $n$ that every $A \in GL_{n+1}(\Re)$ with $A\SOC{p}{n+1}=\SOC{q}{n+1}$ is a bijection on the set $$E=\bigcup_{i=1}^n\bigcup_{\sigma \in \{-1,1\}}\HLS(1,\sigma e_i^n),$$ where $e_i^n$ is the $i$-th standard unit vector in $\Re^n$.
First, let us check that this claim implies $A \in \Aut(\SOC{1}{n+1})$ and $p=q$.
Taking the conical hull of the relation $AE=E$, we conclude that
$$A\SOC{1}{n+1} = A(\conicHull(E)) = \conicHull (AE) = \conicHull(E) = \SOC{1}{n+1},$$
where the relation $\conicHull(E)=\SOC{1}{n+1}$ holds because
a pointed closed convex cone is the conical hull of its extreme rays (see Theorem 18.5 in \cite{rockafellar}) and $E$ is precisely the union of all the extreme rays of $\SOC{1}{n+1}$ with the origin removed, see Section \ref{sec:pcon}. Therefore, we have
$$A \in \Aut(\SOC{1}{n+1}) \subseteq \Aut(\SOC{p}{n+1}),$$
where the last inclusion follows by Proposition~\ref{prop:GT} because $\norm{Px}_p=\norm{x}_p$ for any generalized permutation matrix $P$.
Then $\SOC{p}{n+1}=A\SOC{p}{n+1}=\SOC{q}{n+1}$ and so $p=q$ must hold.

Now, let us show the claim that $A$ is a bijection on $E$.
Consider the map $\xi_p:\Re^n\setminus\{0\} \to M_p$ defined by $\xi_p(x)=(\norm{x}_p,x)$
whose inverse $\xi^{-1}_p:M_p\to \Re^n\setminus\{0\}$ is the projection $\xi^{-1}_p(t,x)=x$.
By Proposition~\ref{prop:C^k},
the map $B:\Re^n\setminus\{0\}\to \Re^n\setminus\{0\}$ defined by
$$B(x)=\xi_q^{-1}\circ A|_{M_p}\circ \xi_p(x)$$
is a $C^1$ diffeomorphism.
Moreover, $\norm{\cdot}_p$ is $C^2$ on a neighborhood of $x$ if and only if $\norm{\cdot}_q$ is $C^2$ on a neighborhood of $B(x)$.
Since $p,q\in (1,2)$, each of the functions $\norm{\cdot}_p$ and $\norm{\cdot}_q$ is $C^2$ on a neighborhood of $x$ if and only if $x_i\ne 0$ for all $i$ (Lemma~\ref{lem:pnorm}). This implies that the set
$$
X=\{x \in \Re^n\setminus\{0\}~|~ x_i = 0\text{ for some } i\}
$$
satisfies
$$
B(X)=X
$$
because $x$ belongs to $X$ if and only if $\norm{\cdot}_p$ and $\norm{\cdot}_q$ are never $C^2$ on any neighborhood of $x$.

\fbox{III.a}
Consider the case $n=2$. Then the set $X$ can be written as
\begin{align*}
X
&= \{x \in \Re^2\setminus\{0\} \mid x_1=0 \text{ or } x_2=0\}
\\
&=\HLS(0,1) \cup \HLS(0,-1) \cup \HLS(1,0) \cup \HLS(-1,0)
\\&= \bigcup_{i=1}^2\bigcup_{\sigma \in \{-1,1\}}\HLS(\sigma e_i^2).
\end{align*}
Then $\xi_p(X)$ and $\xi_q(X)$ coincide with $E$:
$$\xi_p(X)=\xi_q(X)=\bigcup_{i=1}^2 \bigcup_{\sigma \in \{-1,1\}}\HLS(1,\sigma e_i^2)=E.$$
Moreover, $A$ is bijective on $E$  because
$$A(\xi_p(X)) = \xi_q\circ \xi_q^{-1}\circ A|_{M_p}\circ \xi_p(X) = \xi_q\circ B(X)=\xi_q(X).$$
Thus, the claim $AE=E$  holds in the case $n=2$.

\fbox{III.b}
Now let $n\geq 3$ and suppose that the claim is valid for $n-1$.
Denote
$$
X_i := \{x \in \Re^n\setminus\{0\} \mid x_i = 0\},\quad
M_p^i := \xi_p(X_i)= \{(t,x)\in \Re\times \Re^n\setminus\{0\} : t = \norm{x}_p,~x_i=0\}.
$$
With that, we have
$$
X = \bigcup _{i=1}^n X_i.
$$
We show that for any $i\in\{1,\ldots,n\}$ there exists $j\in\{1,\ldots,n\}$ such that
$$
B(X_i)=X_j.
$$
For any $i$,
the set $X_i$ is a connected $(n-1)$ dimensional $C^1$-embedded submanifold of $\Re^n$ contained in $X$.
Since $B:\Re^n\setminus\{0\}\to\Re^n\setminus\{0\}$ is a $C^1$ diffeomorphism satisfying $B(X)=X$,
the set $B(X_i)$ is also a connected $(n-1)$ dimensional $C^1$-embedded submanifold of $\Re^n$ contained in $X$.
Then, since $X\cup\{0\}$ is the union of the hyperplanes $X_i\cup\{0\}$, $i=1,\ldots,n$,
it follows from Proposition~\ref{prop:hplane} that $B(X_i)$ is entirely contained in some hyperplane $X_j\cup\{0\}$.
Then we have $$B(X_i)\subseteq X_j.$$ By the same argument, the set $B^{-1}(X_j)$ is contained in some hyperplane $X_k\cup\{0\}$,
that is, $B^{-1}(X_j)\subseteq X_k$ holds.
This shows that
$$
X_i = B^{-1}(B(X_i)) \subseteq B^{-1}(X_j) \subseteq X_k.
$$
Since $X_i$ cannnot be a subset of $X_k$ if $i\ne k$, it follows that $i=k$.
Then, we obtain $X_i=B^{-1}(X_j)$, i.e., $B(X_i)=X_j$.

Since $B$ is a bijection, the above argument shows that
there exists a permutation $\tau$ on $\{1,\ldots,n\}$ such that
$$B(X_i)=X_{\tau(i)}.$$
Then we have
\[
A(M_p^i) = \xi_q \circ \xi_q^{-1} \circ A|_{M_p} \circ \xi_p (X_i) = \xi_q\circ B(X_i) = \xi_q(X_{\tau(i)}) = M_q^{\tau(i)}.
\]
Taking the linear span both sides, we also have
$$
A(V_i)=V_{\tau(i)} ~\text{ where }~ V_i:=\{(t,x)\in\Re\times\Re^n\mid x_i = 0\}.
$$
Now we apply the induction hypothesis to the isomorphism $A|_{V_i}$ as follows.
Define the isomorphism $\varphi_i:V_i \to \Re^{n}$ by
\[
\varphi_i(t,x_1,\ldots,x_{i-1},0,x_{i+1},\ldots,x_n) = (t,x_1,\ldots,x_{i-1},x_{i+1},\ldots,x_n)
\]
and consider the isomorphism $A_i:=\varphi_{\tau(i)}\circ A|_{V_i} \circ \varphi_i^{-1}:\Re^n\to \Re^n$. By the above argument, we see that $A_i(\SOC{p}{n}) = \SOC{q}{n}$:
\[
A_i(\SOC{p}{n}) = \varphi_{\tau(i)}\circ A|_{V_i} \circ \varphi_i^{-1}(\SOC{p}{n}) = \varphi_{\tau(i)}\circ A(\conicHull M_p^{i})
= \varphi_{\tau(i)}(\conicHull M_q^{\tau(i)}) = \SOC{q}{n}.
\]
So the induction hypothesis implies that
$A_i$ is bijective on
$$\bigcup_{j=1}^{n-1}\bigcup_{\sigma \in \{-1,1\}}\HLS(1,\sigma e_j^{n-1}).$$
Therefore, $A|_{V_i}=\varphi_{\tau(i)}^{-1}\circ A_i^{-1} \circ \varphi_i$ is a bijection from
$$
\bigcup_{j \in \{1,\ldots,n\}\setminus\{i\}}~\bigcup_{\sigma \in \{-1,1\}}\HLS(1,\sigma e_j^n)
$$
onto
$$
\bigcup_{j \in \{1,\ldots,n\}\setminus\{\tau(i)\}} ~\bigcup_{\sigma \in \{-1,1\}} \HLS(1,\sigma e_j^n).
$$
Combining this result for each $i=1,\ldots,n$,
it turns out that
$A$ is bijective on $$E=\bigcup_{i=1}^n\bigcup_{\sigma \in \{-1,1\}}\HLS(1,\sigma e_i^n).$$
\end{proof}

Combining the latter assertion of Theorem~\ref{thm:main} and Proposition~\ref{prop:GT}, we obtain the description of the automorphism group of the $p$-cones.

\begin{corollary}\label{cor:aut}
For $p \in [1,\infty]$, $p \ne 2$ and $n\geq 2$, we have $\Aut(\SOC{p}{n+1}) = \Aut(\SOC{1}{n+1})$.
In particular, any $A \in \Aut(\SOC{p}{n+1})$ can be written as
$$
A=
\alpha\begin{pmatrix}
1 & 0\\
0 & P
\end{pmatrix},
$$
where $\alpha > 0$ and $P$ is an $n\times n$ generalized permutation matrix.
\end{corollary}
We can also recover our previous result  on the non-homogeneity of $p$-cones with $p \ne 2$.
In contrast to \cite{IL17}, here we do not require the theory of $T$-algebras. 
\begin{corollary}\label{cor:hom}
For $p \in [1,\infty]$, $p \ne 2$ and $n\geq 2$, the $p$-cone $\SOC{p}{n+1}$ is not homogeneous.
\end{corollary}
\begin{proof}
By Corollary \ref{cor:aut}, for any $A \in \Aut(\SOC{p}{n+1})=\Aut(\SOC{1}{n+1})$, we have 
that the vector 
$(1,0,\ldots,0)$ is an eigenvector of
$A$.
So, there is no automorphism of $\SOC{p}{n+1}$ that maps $(1,0,\ldots,0)$ to an 
interior point of $\SOC{p}{n+1}$ that 
does not belong to 
$$
\{(\beta,0,\ldots,0 ) \mid \beta > 0\}.
$$
Hence, $\SOC{p}{n+1}$ cannot be homogeneous.
\end{proof}

Now the non-self-duality of $p$-cones $\SOC{p}{n+1}$ for $p\ne 2$ and $n\geq 2$ is an immediate consequence of Theorem~\ref{thm:main} in view of Proposition~\ref{prop:selfdual}, while we need an extra argument for the case $(p,q,n)=(1,\infty,2)$.

\begin{corollary}\label{cor:selfdual}
For $p \in [1,\infty]$, $p \ne 2$ and $n\geq 2$, the $p$-cone $\SOC{p}{n+1}$ is not self-dual under any inner product.
\end{corollary}
\begin{proof}
Suppose that $\SOC{p}{n+1}$ is self-dual under some inner product.
Then, by Proposition~\ref{prop:selfdual}, there exists a symmetric positive definite matrix $A$ such that
$$A\SOC{p}{n+1} = \SOC{q}{n+1}\quad \text{where} \quad \frac{1}{p}+\frac{1}{q}=1.$$
If $(p,q,n)\ne (1,\infty,2), (\infty,1,2)$, then $p=q=2$ must hold by Theorem~\ref{thm:main}.
Now let us consider the case $(p,q,n)=(1,\infty,2)$, i.e., $A\SOC{1}{3} = \SOC{\infty}{3}$.
Recalling \eqref{eq:k1kinf}, we have $B\SOC{1}{3} = \SOC{\infty}{3}$ with
$$
B =
\left(
\begin{array}{ccc}
1 & 0 & 0 \\
0 & \sqrt{2}\cos(\pi/4) & -\sqrt{2}\sin(\pi/4) \\
0 & \sqrt{2}\sin(\pi/4) & \sqrt{2}\cos(\pi/4)
\end{array}
\right)
= \left(
\begin{array}{ccc}
1 & 0 & 0 \\
0 & 1 & -1 \\
0 & 1 & 1
\end{array}
\right).
$$
Therefore, $B^{-1}A \in \Aut(\SOC{1}{3})$ holds. Then, by Proposition~\ref{prop:GT}, the matrix $A$ can be written as $A = B C$ where $C$ is of the form
$$
C = \alpha \left(
\begin{array}{ccc}
1 & 0 & 0 \\
0 & \pm 1 & 0 \\
0 & 0 & \pm 1
\end{array}
\right)
~~\text{or}~~
\alpha \left(
\begin{array}{ccc}
1 & 0 & 0 \\
0 & 0 & \pm 1 \\
0 & \pm 1 & 0
\end{array}
\right),\quad \alpha >0.
$$
Since $A$ is symmetric, it has one of the following forms:
$$
\alpha \left(
\begin{array}{ccc}
1 & 0 & 0 \\
0 & -1 & -1 \\
0 & -1 & 1
\end{array}
\right),\quad
\alpha \left(
\begin{array}{ccc}
1 & 0 & 0 \\
0 & 1 & 1 \\
0 & 1 & -1
\end{array}
\right),\quad
\alpha \left(
\begin{array}{ccc}
1 & 0 & 0 \\
0 & -1 & 1 \\
0 & 1 & 1
\end{array}
\right),\quad
\alpha \left(
\begin{array}{ccc}
1 & 0 & 0 \\
0 & 1 & -1 \\
0 & -1 & -1
\end{array}
\right).
$$
None of them is positive definite.
Therefore, we obtain a contradiction.
\end{proof}

{\small
	\section*{Acknowledgements}
This work was partially supported by the Grant-in-Aid for Scientific Research (B) (18H03206) and the Grant-in-Aid for Young Scientists (B) (17K12645) from Japan Society for the Promotion of Science.
}

\bibliographystyle{abbrvurl}
\bibliography{bib}

\end{document}